\documentclass{amsart}

\usepackage{a4wide,amsmath,amssymb}
\usepackage[toc,page]{appendix}

\usepackage{url}
\usepackage{hyperref}
\usepackage{xcolor}

\newtheorem{thm}{Theorem}[section]
\newtheorem{lemma}[thm]{Lemma}

\newtheorem{cor}[thm]{Corollary}
\theoremstyle{remark}
\newtheorem{rem}[thm]{Remark}

\theoremstyle{definition}
\newtheorem{defn}[thm]{Definition}
\newtheoremstyle{Claim}{}{}{\itshape}{}{\itshape\bfseries}{:}{ }{#1}
\theoremstyle{Claim}

\newcommand{\R}{\mathbb{R}}

\newcommand{\X}{\mathcal{X}}

\newcommand{\eps}{\varepsilon}

\theoremstyle{plain}

\def\sideremark#1{\ifvmode\leavevmode\fi\vadjust{
\vbox to0pt{\hbox to 0pt{\hskip\hsize\hskip1em
\vbox{\hsize3cm\tiny\raggedright\pretolerance10000
\noindent #1\hfill}\hss}\vbox to8pt{\vfil}\vss}}}

\begin{document}

\title[]{H\"older regularity and Liouville properties for nonlinear elliptic inequalities with power-growth gradient terms}


%
\author{Alessandro Goffi}
\address{Dipartimento di Matematica ``Tullio Levi-Civita'', Universit\`a degli Studi di Padova, 
via Trieste 63, 35121 Padova (Italy)}
\curraddr{}
\email{alessandro.goffi@unipd.it}
\thanks{}

\subjclass[2010]{Primary: 35B53, 35B65, 35J62}
\keywords{Degenerate equations, Equations of divergence-type, H\"older regularity, Nonexistence of entire weak solutions, quasilinear Hamilton-Jacobi equations, Riemannian manifolds}
 \thanks{
 The author is member of the Gruppo Nazionale per l'Analisi Matematica, la Probabilit\`a e le loro Applicazioni (GNAMPA) of the Istituto Nazionale di Alta Matematica (INdAM)
 }

\date{\today}

\begin{abstract}
This note studies local integral gradient bounds for distributional solutions of a large class of partial differential inequalities with diffusion in divergence form and power-like first-order terms. The applications of these estimates are two-fold. First, we show the (sharp) global H\"older regularity of distributional semi-solutions to this class of diffusive PDEs with first-order terms having supernatural growth and right-hand side in a suitable Morrey class posed on a bounded and regular open set $\Omega$. Second, we provide a new proof of entire Liouville properties for inequalities with superlinear first-order terms without assuming any one-side bound on the solution for the corresponding homogeneous partial differential inequalities. We also discuss some extensions of the previous properties to problems arising in sub-Riemannian geometry and also to partial differential inequalities posed on noncompact complete Riemannian manifolds under appropriate area-growth conditions of the geodesic spheres, providing new results in both these directions. The methods rely on integral arguments and do not exploit maximum and comparison principles.
\end{abstract}

\maketitle


\section{Introduction}

In this note we analyze some quantitative and qualitative results for the following partial differential inequality (PDI in the sequel)
\begin{equation}\label{eqgen}
-\mathrm{div}(\mathcal{A}(x,u,\nabla u))\geq \mathcal{B}_i(x,u,\nabla u)\text{ in }\Omega\ ,
\end{equation}
where $\Omega$ will be either a bounded open set in $\R^N$ or the whole space itself, $N> 1$,  $\mathcal{A}:\Omega\times\R\times\R^N\to\R^N$ is a Carath\'eodory function (namely measurable in $x$ and continuous in the $u,\nabla u$ entries) such that
\begin{equation}\label{A}\tag{A}
|\mathcal{A}(x,s,\xi)|\leq \nu|\xi|^{p-1}\ ,\nu>0\ ,
\end{equation}
for every $(s,\xi)\in\R\times\R^N$, a.e. $x\in\Omega$, and
\begin{equation}\label{B1}\tag{B1}
\mathcal{B}_1(x,s,\xi)\geq c_H|\xi|^\gamma-\lambda s-f(x),
\end{equation}
\begin{equation}\label{B2}\tag{B2}
\mathcal{B}_2(x,s,\xi)\geq c_H|\xi|^\gamma.
\end{equation}
Here $\mathcal{B}_1:\Omega\times\R\times\R^N\to\R$, $\mathcal{B}_2:\R\times\R^N\to\R$ , $\lambda\geq0$, $c_H>0$ and $f$ is a measurable source term belonging to $\mathcal{L}^{1,\frac{N}{q}}(\Omega)$, the Morrey space of parameters $(1,N/q)$, $q\geq1$, $\gamma>p-1$, $p>1$. A prototype example of an equation satisfying the previous assumptions is the quasilinear Hamilton-Jacobi/Riccati equation driven by the $p$-Laplacian \cite{Lions85,LP,VeronJFA,CGell,CGpar}
\[
-\mathrm{div}(|\nabla u|^{p-2}\nabla u)+\lambda u+|\nabla u|^\gamma=g(x)\ .
\]
Our main results are the following. The first is a quantitative result for \textit{semi-solutions} to PDIs, that is summarized in the next
\begin{thm}\label{main1}
Let $\Omega$ be an open bounded and connected subset of $\R^N$ having Lipschitz boundary and satisfying the uniform interior sphere condition. Assume \eqref{A} and \eqref{B1}, $\gamma>p$, $\lambda\geq0$ and $f\in \mathcal{L}^{1,\frac{N}{q}}(\Omega)$ for some $q>\frac{N}{\gamma}$. Let $u\in W^{1,\gamma}_{\mathrm{loc}}(\Omega)$  such that $\lambda u^+\in \mathcal{L}^{1,\frac{N}{q}}(\Omega)$ which satisfies, in the sense of distributions, the inequality 
\[
-\mathrm{div}(\mathcal{A}(x,u,\nabla u))\geq \mathcal{B}_1(x,u,\nabla u)\text{ in }\Omega\ .
\] 
Then $u$ is H\"older continuous up to the boundary (namely in the whole $\overline{\Omega}$) and satisfies
\[
|u(x)-u(y)|\leq K|x-y|^\alpha\ ,\forall x,y\in\overline{\Omega},
\]
where
\[
\alpha=\min\left\{1-\frac{N}{q\gamma},\frac{\gamma-p}{\gamma-(p-1)}\right\},
\]
and $K$ depends on $p,q,\gamma,N,\nu,\Omega, \|f+\lambda u^+\|_{\mathcal{L}^{1,\frac{N}{q}}(\Omega)}$. In particular, the estimate is satisfied when $f,\lambda u^+\in L^q(\Omega)$, $q>\frac{N}{\gamma}$.
\end{thm}
The second is a Liouville/Bernstein-type result and states the nonexistence of nontrivial solutions to \eqref{eqgen} for the corresponding homogeneous PDE/PDI without zero-th order terms on noncompact complete Riemannian manifolds $M$ in terms of \textit{the growth rate of the area of the geodesic spheres}.
\begin{thm}\label{main3}
Let $(M,g)$ be a noncompact complete Riemannian manifold. If $p>1$, $\gamma>p-1$ and
\begin{equation}\label{areanl}
\int^{+\infty}\frac{1}{(\mathrm{area}(\partial B_t(o)))^{\frac{\gamma-(p-1)}{p-1}}}\,dt=+\infty
\end{equation}
for some origin $o\in M$, then any distributional solution of
\begin{equation}\label{hjmanifold}
-\mathrm{div} (\mathcal{A}(x,u,\nabla u))\geq \mathcal{B}_2(x,u,\nabla u)\text{ in }M,
\end{equation}
where $\mathcal{A}$ satisfies \eqref{A}, must be a.e. constant on $M$. Similarly, if \eqref{areanl} holds, any distributional solution to  
\begin{equation}\label{hjmanifold}
-\mathrm{div} (\mathcal{A}(x,u,\nabla u))+ c_H|\nabla u|^\gamma=0\text{ in }M
\end{equation}
with $\gamma>p-1$ must be a.e. constant on $M$.
\end{thm}
The previous result leads to the following property when specialized to the Euclidean space $M=\R^N$.
\begin{cor}\label{main2}
Let $\mathcal{A}$ be satisfying \eqref{A}. Any distributional solution to the inequality
\begin{equation}\label{eqmain1}
-\mathrm{div}(\mathcal{A}(x,u,\nabla u))\geq \mathcal{B}_2(x,u,\nabla u) \text{ in }\R^N
\end{equation}
must be a.e. constant in $\R^N$ provided that
\begin{equation}\label{rangegamma}
p-1<\gamma\leq \frac{N(p-1)}{N-1}\ ,1<p<N\ .
\end{equation}
Similarly, any distributional solution to  \begin{equation}\label{hjeq}
-\mathrm{div} (\mathcal{A}(x,u,\nabla u))+c_H|\nabla u|^\gamma=0\text{ in }\R^N
\end{equation}
must be a.e. constant when $\gamma,p$ vary in the range \eqref{rangegamma}.
\end{cor}
These results rely on interior integral gradient estimates of Caccioppoli-type of the form
\[
\int_{B_r}|\nabla u|^\gamma\,dx\leq CR^\alpha
\]
for some $\alpha\in\R$, see Lemma \ref{cac}, Remark \ref{simp} and Lemma \ref{prelman}. Such bounds are obtained either by simple integral arguments that exploit test function methods or combine a weak version of the divergence theorem and the co-area formula.\\
Related properties to those in Theorem \ref{main1} for problems with supernatural gradient growth have been already analyzed in the literature. The paper by I. Capuzzo-Dolcetta, F. Leoni and A. Porretta \cite{CDLP} treated a large class of diffusion operators in non-divergence form (that also includes fully nonlinear uniformly elliptic operators and the $p$-Laplacian) with bounded right-hand side $f$ (with estimates depending on $\|u\|_\infty$) in the realm of viscosity solutions. This was revisited with a shorter proof by G. Barles in \cite{B}, again in the framework of viscosity (semi-)solutions. This regularity problem has been much less investigated for unbounded data, where mostly optimal regularity at the level of Lebesgue spaces has been studied, see e.g. \cite{GMP}. An analogue to our Theorem \ref{main1} with unbounded source terms $f\in L^q$ was proved by A. Dall'Aglio and A. Porretta in the case $p=2$ for distributional subsolutions to equations with terms having opposite signs, and it has been the starting point of our work. At this stage, we stress that distributional solutions for these problems are not unique, cf Section \ref{sharp}, and hence the quantitative result in Theorem \ref{main1} takes on a stronger meaning.\\
Recent advances on the H\"older regularity for problems driven by the Laplacian with coercive gradient terms and space-time $L^q$ right-hand side have been obtained for parabolic equations in \cite{CGpar} through quite different integral/duality methods, which, unfortunately, do not extend to quasi-linear operators satisfying \eqref{A}. We further mention that the H\"older regularity for solutions to these equations when $p=2$ has been also the focus of the recent work \cite{CV}, but for solutions solving the PDE in strong sense, looking thus beyond the range of the H\"older bound stated in Theorem \ref{main1} for source terms with better summability. Indeed, we also show that, as far as the H\"older regularity is concerned, the range found in Theorem \ref{main1} is sharp for our class of functions, cf Section \ref{sharp}. It is well-known that these regularity properties are the cornerstone for the analysis of many different problems ranging from those arising in ergodic control and homogenization to problems with state constraints, at least when $p=2$, see \cite{B,BDL}. Nonetheless, the case of quasilinear equations with terms having supernatural growth appears, to our knowledge, still an open line of research, cf \cite{LP,BDL,NP} for some recent contributions.\\
Instead, statements like Theorem \ref{main3}, and its byproduct Corollary \ref{main2}, have been studied (almost) thoroughly and for many years through various different approaches, see e.g. Remark \ref{ref}, Section \ref{alt} and the general reference \cite{PucciBook}. Although the latter is a widely studied problem, we decided to propose a new unifying proof of the result, even in a more general context, due to its shortness, along with the treatment of some new endpoint cases (e.g. the borderline case $\gamma=\frac{N(p-1)}{N-1}$ in Corollary \ref{main2}). Here, our approach to tackle the Liouville property has been inspired by the works \cite{Lions85,RS,PRSmem}. At this point it is remarkable to emphasize that when $p=2$ the threshold $\gamma=\frac{N}{N-1}$ is critical for the solvability and the regularity of solutions of the so-called viscous Hamilton-Jacobi equation, see e.g. \cite{HMV99,CGell}.\\
Notably, the results in Theorems \ref{main1}, \ref{main3} and Corollary \ref{main2} highlight a striking effect of the superlinear gradient term. Indeed, a simple zoom of the equation suggests that such a term is the dominating one at small scales when $\gamma>p$ (i.e. in the regime of Theorem \ref{main1}), and it turns out to be the sole responsible of all the properties analyzed throughout this manuscript, as one can realize by a careful inspection of the proofs of Theorems \ref{main1} and \ref{main3}. The link between the Liouville and the local regularity properties of nonlinear elliptic problems has been sometimes pointed out throughout some works, see e.g. \cite{MeierTAMS,BBF} among others, and our analysis underlines once more the bridge between nonexistence results and the H\"older regularity for semi-solutions of PDEs through the Caccioppoli-type bounds in Lemma \ref{cac} and Remark \ref{simp}. \\

Some comments on the aforementioned results are now in order. On one hand, Theorem \ref{main1} holds for merely \textit{distributional semi-solutions} to PDEs with \textit{supernatural gradient growth} and \textit{unbounded right-hand side}. This is unusual for diffusive problems (e.g. of second order) with first-order terms above the natural growth imposed by the diffusion, since such estimates hold, in general, for \textit{solutions} of \textit{uniformly elliptic problems} with \textit{subnatural gradient growth} and, usually, \textit{bounded data} \cite{LU,BensoussanBook}. Moreover, the estimate in Theorem \ref{main1} holds up to the boundary of the domain and it is a universal estimate for negative solutions.\\
On the other hand, the approach leading to Theorem \ref{main3} does not require any one-side bound on the solution, being thus consistent with the companion results for solutions to the same kind of equations obtained through the Bernstein gradient estimates by L. Peletier-J. Serrin \cite{PS} and P.-L. Lions \cite{Lions85}, and later in \cite{VeronJFA} (see also \cite[Theorem 2.4]{BDL} and \cite[Theorem 1.1]{FPS} for recent refinements). More importantly, in the context of Riemannian manifolds Theorem \ref{main3} does not require curvature conditions on the background geometry. As a matter of fact, we emphasize that curvature bounds have been frequently imposed to derive the Liouville property for solutions to $-\Delta_pu+|Du|^\gamma=0$ when the first-order term has superlinear growth with respect to the diffusion (i.e. $\gamma>p-1$) and the equation is posed on a noncompact manifold, see e.g. \cite[Corollary 4.3]{VeronJFA}. This is a consequence of the use of Bernstein-type methods through the B\"ochner's identity. In this paper we adopt a different viewpoint, as in \cite{RS,PRSjfa}. Indeed, Theorem \ref{main3} shows, in the case of distributional supersolutions to these homogeneous quasi-linear equations, that the Liouville property can be obtained under some area-growth conditions, at least for slowly increasing gradient terms, see Section \ref{sec;riem}. It seems an open problem whether the Liouville result in \cite{VeronJFA} holds for less regular solutions and with less restrictive bounds on the geometry in the general superlinear range $\gamma>p-1$. We further remark that the assumptions we impose in Theorem \ref{main3} are closer to the ones used in \cite{RS,Sun} rather than in \cite{VeronJFA}.\\
Our results apply to general diffusion operators in divergence form. For instance $-\mathrm{div}(\mathcal{A}(x,u,\nabla u))$ can be one of the following:
\begin{itemize}
\item the (negative) Laplacian $-\Delta u=-\mathrm{div}(\nabla u)$ or the mean-curvature operator $-\mathrm{div}\left(\frac{\nabla u}{\sqrt{1+|\nabla u|^2}}\right)$ when $p=2$;
\item the $p$-Laplacian $-\Delta_pu=-\mathrm{div}(|\nabla u|^{p-2}\nabla u)$ for $p>1$;
\item the generalized mean-curvature operator $-\mathrm{div}\left(|\nabla u|^{k-2}\frac{\nabla u}{\sqrt{1+|\nabla u|^k}}\right)$, $k\geq2$, for $p=\frac{k}{2}$.
\end{itemize}
Moreover, the properties in Theorems \ref{main1} and \ref{main2} also extend to the subelliptic framework when the standard derivatives in the aforementioned operators are replaced with derivatives along vector fields satisfying the H\"ormander's rank condition, i.e. for the subelliptic inequality
\[
-\mathrm{div}(\mathcal{A}(x,u,\nabla_\X u))\geq |\nabla_\X u|^\gamma-f(x),
\]
where $\nabla_\X$ stands for the horizontal gradient built over the frame $\X$, see e.g. Theorem \ref{holderHorm} in Section \ref{sec;sub} for more details. This is the case, for instance, of problems structured over the fields generating a Carnot group of step 2, the main prototype being the Heisenberg group. The interior H\"older regularity for such subelliptic problems with power-like first-order terms on the horizontal gradient seems new to our knowledge.\\
Moreover, the results in Theorems \ref{main1} and \ref{main2} can be stated analogously for distributional subsolutions, as in e.g. \cite{DP} for the case $p=2$, by reversing the signs in the assumptions \eqref{B1} and \eqref{B2}. For instance, the conclusions of Theorem \ref{main1} hold for distributional subsolutions to the model PDI
\[
-\mathrm{div}(\mathcal{A}(x,u,\nabla u))+\lambda u+|\nabla u|^\gamma\leq f(x)\text{ in }\Omega\ ,
\]
under essentially the same regularity assumptions on the data, while the Liouville property in Theorem \ref{main2} holds for the weak PDI
\[
-\mathrm{div}(\mathcal{A}(x,u,\nabla u))+c_H|\nabla u|^\gamma\leq0\text{ in }\R^N
\]
under the hypothesis \eqref{A}. \\
We conclude by saying that our underlying aim here is also to lay the groundwork to investigate the problem of the optimal gradient regularity in Lebesgue spaces for equations with diffusion in divergence form as those modeled over the $p$-Laplacian, as initiated in \cite{CGell,CGpar} for the viscous Hamilton-Jacobi equation, through blow-up arguments, see \cite{BBF,CV}.\\
\par\smallskip
\textit{Plan of the paper}. Section \ref{sec;cac} introduces some preliminary definitions and contains the main integral bounds needed for the following sections. Section \ref{sec;hol} studies the interior and global H\"older regularity for semi-solutions, while Section \ref{sec;lio} the Liouville property for homogeneous PDIs in the Euclidean space, that is Corollary \ref{main2}. Section \ref{sec;sub} and \ref{sec;riem} conclude the paper with the local H\"older regularity for solutions to quasi-linear subelliptic inequalities and the Liouville properties for inequalities posed on Riemannian manifolds.
\section{Caccioppoli-type inequalities}\label{sec;cac}
We begin with some preliminary definitions and notations.
\begin{defn}
Let $\Omega$ be either a bounded open set or the whole Euclidean space $\R^N$. We say that $u\in W^{1,p}_{\text{loc}}(\Omega)$  is a distributional supersolution of \eqref{eqgen} if $\mathcal{A}(\cdot,u,\nabla u)\in L^{p'}_{\mathrm{loc}}(\Omega)$ with $\mathcal{B}(\cdot,u,\nabla u)\in L^1_{\mathrm{loc}}(\Omega)$ and
\[
\int_\Omega \mathcal{A}(x,u,\nabla u)\cdot \nabla \varphi\,dx\geq \int_\Omega  \mathcal{B}(x,u,\nabla u)\varphi\,dx\ ,\varphi\in C_0^1(\Omega)\ ,\varphi\geq0\ .
\]
\end{defn}
\begin{rem}
When $\gamma>p$ as in Theorem \ref{main1}, it will be enough to consider distributional supersolutions belonging to $u\in W^{1,\gamma}_{\text{loc}}(\Omega)$ in view of the inclusions of Lebesgue spaces.
\end{rem}
We set $u^+=\max\{u,0\}$ and $u^-=\max\{-u,0\}$. If $r\in(1,\infty)$, we denote by $r'=\frac{r}{r-1}$ its H\"older conjugate exponent. Recall that for $s\geq1$ and $\theta\in(0,N]$, the Morrey space $\mathcal{L}^{s,\theta}(\Omega)$  comprises those functions $h\in L^s(\Omega)$ such that
\[
\int_{B_r(z)\cap\Omega}|h(x)|^s\leq Cr^{N-\theta}
\]
for all $z\in\Omega$, $r\in(0,\mathrm{diam}(\Omega)]$ and $C$ independent of $z$ and $r$. This space is equipped with the norm
\[
\|h\|_{\mathcal{L}^{s,\theta}(\Omega)}:=\sup_{z\in\Omega;0<r\leq \mathrm{diam}(\Omega)} r^{\frac{\theta-N}{s}}\|h\|_{L^s(B_r(z)\cap\Omega)}.
\]
Clearly, if $s=1$ and $\theta=N/q$, $q\geq1$, any function $h\in L^q(\Omega)$ satisfies the bound
\[
\int_{B_r(z)\cap\Omega}|h(x)|\leq Cr^{N-\frac{N}{q}}
\]
by the H\"older's inequality. For the properties that will be stated in the Euclidean setting, $\mathcal{H}^{N-1}$ stands for the ($N-1$)-Hausdorff measure, $B_r(x)$ is the Euclidean ball with center $x$ and radius $r$, hence $\mathrm{area}(\partial B_r(x))=N\omega_N r^{N-1}$ and $\mathrm{Vol}(B_r(x))=\omega_Nr^N$, where $\omega_N$ stands for the measure of the unit Euclidean ball. \\
The main result of this section is the following Caccioppoli estimate.
\begin{lemma}\label{cac}
Let $\gamma>p-1$, $\lambda\geq0$, $f\in \mathcal{L}^{1,\frac{N}{q}}(\Omega)$, $q\geq1$. Assume that \eqref{A} and \eqref{B1} hold. Let $u$ be a distributional supersolution of \eqref{eqgen} such that $\lambda u^+\in \mathcal{L}^{1,\frac{N}{q}}(\Omega)$. Then, for every pair of concentric balls $B_t\subset B_R\subset \Omega$ we have
\[
\int_{B_t}|\nabla u|^\gamma\,dx+\lambda\int_{B_t}u^-\,dx\leq K\frac{R^N}{(R-t)^s}
\]
where $s=\max\left\{\frac{N}{q},\frac{\gamma}{\gamma-(p-1)}\right\}$ and $K$ is a constant depending on $\nu,p,\gamma,q,N,c_H$ and on $\|f+\lambda u^+\|_{\mathcal{L}^{1,\frac{N}{q}}(B_R)}$. 

\end{lemma}
\begin{proof}
Let $C$ be a generic constant depending on the data $\nu,p,\gamma,q,N,c_H$. Let $\eta$ be a $C^1$ cut-off function such that $0\leq\eta\leq1$, $\eta\equiv1$ on $B_t$, $\eta\equiv0$ outside $B_R$, $|\nabla \eta|\leq \frac{C}{R-t}$. 
We use $\varphi=\eta^\frac{\gamma}{\gamma-(p-1)}$ as a test function in the distributional formulation of the inequality \eqref{eqgen}, together with \eqref{B1}, to obtain
\begin{multline*}
c_H\int_\Omega |\nabla u|^\gamma\eta^\frac{\gamma}{\gamma-(p-1)}\,dx+\lambda \int_\Omega u^-\eta^\frac{\gamma}{\gamma-(p-1)}\,dx\\
\leq \frac{\gamma}{\gamma-(p-1)}\int_\Omega (\mathcal{A}(x,u,\nabla u)\cdot \nabla \eta)\eta^{\frac{p-1}{\gamma-(p-1)}}\,dx+\int_\Omega (f+\lambda u^+)\eta^\frac{\gamma}{\gamma-(p-1)}\,dx.
\end{multline*}
Observe that by the weighted Young's inequality with exponents $(\frac{\gamma}{\gamma-(p-1)},\frac{\gamma}{p-1})$ we have
\begin{multline*}
\frac{\gamma}{\gamma-(p-1)}\int_\Omega (\mathcal{A}(x,u,\nabla u)\cdot \nabla \eta)\eta^{\frac{p-1}{\gamma-(p-1)}}\,dx\leq \frac{\gamma\nu}{\gamma-(p-1)}\int_\Omega |\nabla u|^{p-1}|\nabla\eta|\eta^{\frac{p-1}{\gamma-(p-1)}}\,dx\\
\leq \frac{c_H}{2}\int_\Omega |\nabla u|^\gamma\eta^\frac{\gamma}{\gamma-(p-1)}\,dx+C(c_H,\gamma,p,\nu)\int_\Omega |\nabla\eta|^{\frac{\gamma}{\gamma-(p-1)}}\,dx.
\end{multline*}
We then get
\begin{align*}
\frac{c_H}{2}&\int_\Omega |\nabla u|^\gamma\eta^\frac{\gamma}{\gamma-(p-1)}\,dx+\lambda \int_\Omega u^-\eta^\frac{\gamma}{\gamma-(p-1)}\,dx \\
&\leq \int_\Omega (f+\lambda u^+)\eta^\frac{\gamma}{\gamma-(p-1)}\,dx+C(c_H,\gamma,p,\nu)\int_\Omega |\nabla\eta|^{\frac{\gamma}{\gamma-(p-1)}}\,dx\\
&\leq \omega_N\|f+\lambda u^+\|_{\mathcal{L}^{1,\frac{N}{q}}(B_R)}R^{N-\frac{N}{q}}+C(c_H,\gamma,p,\nu,\omega_N)\frac{R^N}{(R-t)^\frac{\gamma}{\gamma-(p-1)}}\leq K\frac{R^N}{(R-t)^s},
\end{align*}
where $s=\max\left\{\frac{N}{q},\frac{\gamma}{\gamma-(p-1)}\right\}$ and $K$ depends on $N,c_H,\gamma,\nu,p,q$ and $ \|f+\lambda u^+\|_{\mathcal{L}^{1,\frac{N}{q}}(B_R)}$.
\end{proof}
\begin{rem}
The estimate in Lemma \ref{cac} holds for any $\gamma>p-1$, $p>1$, and hence slightly improves the one obtained in \cite{DP} for the case $p=2$, being new in the regime $1<\gamma\leq 2$. This will be important to derive the Liouville theorems for PDIs with nonlinearities having subnatural growth in the gradient in the next section.
\end{rem}
\begin{rem}\label{simp}
The proof of the previous estimate simplifies when \eqref{B1} is replaced with \eqref{B2}, i.e. the term on the right-hand side of the PDI \eqref{eqgen} satisfies $\mathcal{B}_2(x,u,\nabla u)\geq c_H|\nabla u|^\gamma$ using a different integral approach, inspired by \cite{Lions85,RS}. This will be crucial to derive the Liouville results in the next section. Suppose that $u$ is a nonconstant distributional solution to \eqref{eqgen}. We proceed formally integrating the PDI to highlight the main ingredients, although the argument can be made rigorous reasoning as in Remark \ref{weakdiv}. In this way, we get
\begin{equation}\label{eq1}
-\int_{B_r}\mathrm{div}(\mathcal{A}(x,u,\nabla u))\,dx\geq c_H\int_{B_r}|\nabla u|^\gamma\,dx,
\end{equation}
 to conclude after applying the H\"older's inequality
\begin{align}\label{eq2}
-\int_{B_r}\mathrm{div}(\mathcal{A}(x,u,\nabla u))\,dx&\leq \nu\int_{\partial B_r}|\nabla u|^{p-1}\,d\mathcal{H}^{N-1}\\
&\leq \nu\left(\int_{\partial B_r}|\nabla u|^{\gamma}\,d\mathcal{H}^{N-1}\right)^{\frac{p-1}{\gamma}}(\mathrm{area}(\partial B_r))^{\frac{\gamma-(p-1)}{\gamma}}.
\end{align}
We set $\sigma(r)=\int_{B_r}|\nabla u|^\gamma$. Note that since $u$ is non-constant, there exists $\overline{R}$ such that for $r\geq \overline{R}$ it results $\sigma(r)>0$. Moreover, by the co-area formula it follows that \[\sigma'(r)=\int_{\partial B_r}|\nabla u|^\gamma\,d\mathcal{H}^{N-1}.\]
Combining \eqref{eq1} and \eqref{eq2} we conclude
\[
\nu[\sigma'(r)]^{\frac{p-1}{\gamma}}(\mathrm{area}(\partial B_r))^{\frac{\gamma-(p-1)}{\gamma}}\geq c_H\sigma(r)
\]
and so
\[
\sigma'(r)\geq C(\nu,N,\gamma,p,c_H)(\sigma(r))^{\frac{\gamma}{p-1}}r^{-(N-1)\frac{\gamma-(p-1)}{p-1}}, r\geq\overline{R}.
\]
We set 
\[
\alpha:=(N-1)\frac{\gamma-(p-1)}{p-1}
\]
and integrate on $[R,r]$, $R\geq \overline{R}$. When $\alpha< 1$, i.e. $\gamma<\frac{N(p-1)}{N-1}$, we obtain that
\begin{multline}\label{alpha<1}
\frac{p-1}{\gamma-(p-1)}\cdot\frac{1}{(\sigma(R))^{\frac{\gamma-(p-1)}{p-1}}}\geq \frac{p-1}{\gamma-(p-1)}\left(\frac{1}{(\sigma(R))^{\frac{\gamma-(p-1)}{p-1}}}-\frac{1}{(\sigma(r))^{\frac{\gamma-(p-1)}{p-1}}}\right)\\
\geq C(\nu,N,\gamma,p,c_H)\int_R^r t^{-(N-1)\frac{\gamma-(p-1)}{p-1}}=C(\nu,N,\gamma,p,c_H)\frac{(p-1)}{N(p-1)-(N-1)\gamma}[r^{1-\alpha}-R^{1-\alpha}].
\end{multline}
which is essentially the same bound found in Lemma \ref{cac} with $f=\lambda=0$. Remarkably, when $\alpha=1$, i.e. $\gamma=\frac{N(p-1)}{N-1}$, we get
\begin{equation}\label{alpha1}
\frac{p-1}{\gamma-(p-1)}\frac{1}{(\sigma(R))^{\frac{\gamma-(p-1)}{p-1}}}\geq C(\nu,N,\gamma,p,c_H)\log\left(\frac{r}{R}\right).
\end{equation}
This procedure will be further generalized in Lemma \ref{prelman} in the context of Riemannian manifolds.\\
\end{rem}
\begin{rem}\label{weakdiv}
The computations in Remark \ref{simp} can be made rigorous arguing as in \cite{RS,PRSjfa,PSNotes,Chen}. We first recall that $u$ is a distributional supersolution of \eqref{eqgen} if
\[
\int_{\R^N}\mathcal{A}(x,u,\nabla u)\cdot \nabla\varphi\,dx\geq \int_{\R^N}\mathcal{B}(x,u,\nabla u)\varphi\,dx
\]
for all Lipschitz continuous functions $\varphi\geq0$ with compact support. As already outlined in Remark \ref{simp}, the idea is to apply the divergence theorem to the vector field $\mathcal{A}$. To this aim, we define the Lipschitz function $\psi_\eps=\psi_{r,\eps}$ 
\[
\psi_{R,\eps}=
\begin{cases}
1&\text{ if }|x|\leq r,\\
\frac{r+\eps-|x|}{\eps}&\text{ if }r<|x|<r+\eps,\\
0&\text{ if }|x|\geq r+\eps.
\end{cases}
\]
We then take a Lipschitz continuous function $\rho\geq0$ with compact support to be defined later and use $\varphi=\rho\psi_\eps$ as a test function, together with \eqref{A}, to find
\begin{align*}
\int_{\R^N}\mathcal{A}(x,u,\nabla u)\cdot \psi_\eps\nabla\rho\,dx&=\int_{\R^N}\mathcal{A}(x,u,\nabla u)\cdot \nabla(\psi_\eps\rho)\,dx-\int_{\R^N}\mathcal{A}(x,u,\nabla u)\cdot \rho\nabla(\psi_\eps)\,dx\\
&\geq  \int_{\R^N}\mathcal{B}(x,u,\nabla u)\rho\psi_\eps\,dx-\frac{1}{\eps}\int_{B_{r+\eps}/B_r}|\mathcal{A}(x,u,\nabla u)|\rho\,dx\\
&\geq c_H\int_{\R^N}|\nabla u|^\gamma\rho\psi_\eps\,dx-\frac{\nu}{\eps}\int_{B_{r+\eps}/B_r}|\nabla u|^{p-1}\rho\,dx.
\end{align*}
We then choose $\rho$ such that $\rho=1$ on $\overline{B}_{r+\eps}$ and find
\[
\int_{\R^N}\mathcal{A}(x,u,\nabla u)\cdot \psi_\eps\nabla\rho\,dx\geq c_H\int_{B_r}|\nabla u|^\gamma\,dx-\frac{\nu}{\eps}\int_{B_{r+\eps}/B_r}|\nabla u|^{p-1}\,dx.
\]
Therefore, the leftmost side integral of the above inequality vanishes, so that we end up with the inequality 
\[
c_H\int_{B_r}|\nabla u|^\gamma\,dx\leq \frac{\nu}{\eps}\int_{B_{r+\eps}/B_r}|\nabla u|^{p-1}\,dx.
\]
Then, owing to the co-area formula \cite{EG} we get
\begin{multline*}
\lim_{\eps\to0^+}\frac{\nu}{\eps}\int_{B_{r+\eps}/B_r}|\nabla u|^{p-1}\,dx
=\lim_{\eps\to0^+}\frac{\nu}{\eps}\int_{0}^\eps \int_{\partial B_r}|\nabla u|^{p-1}\,d\mathcal{H}^{N-1}\,dr=\nu\int_{\partial B_r}|\nabla u|^{p-1}\,d\mathcal{H}^{N-1}.
\end{multline*}
Letting $\eps\to0^+$ we finally obtain 
\[
c_H\int_{B_r}|\nabla u|^\gamma\,dx\leq \nu\int_{\partial B_r}|\nabla u|^{p-1}\,d\mathcal{H}^{N-1}.
\]
We can then proceed through the H\"older's inequality on the boundary integral as in Remark \ref{simp} to conclude the statement.
\end{rem}
\begin{rem}
One can also prove a similar property when the diffusion operator has the opposite sign, i.e. it is of the form $+\mathrm{div}(\mathcal{A}(x,u,\nabla u))$, and when $|\mathcal{A}(x,u,\nabla u)|\leq h(x)+\nu|\xi|^{p-1}$ with $h\in L^r(\Omega)$ (or in a suitable Morrey class) for some $r>1$ with a (possibly) different H\"older exponent.
\end{rem}
\begin{rem}
A similar condition to \eqref{alpha<1} regarding the existence of solutions has been already pointed out in \cite{HMV99,NP}.
\end{rem}
\section{Some applications of the Caccioppoli-type inequalities}
\subsection{Local and global H\"older regularity for semi-solutions}\label{sec;hol}
We now apply Lemma \ref{cac} to prove the local H\"older continuity of semi-solutions to non-homogeneous PDEs with power-growth nonlinearities and unbounded right-hand side controlled in Morrey spaces. We will actually give a uniform estimate for the H\"older semi-norm on any ball $B\subset \Omega$. 
\begin{thm}\label{localHolder}
Assume  \eqref{A}-\eqref{B1}, $\gamma>p$, $\lambda\geq0$, with $f\in \mathcal{L}^{1,\frac{N}{q}}(\Omega)$ for some $q>\frac{N}{\gamma}$. Let $u\in W^{1,\gamma}_{\mathrm{loc}}(\Omega)$  such that $\lambda u^+\in \mathcal{L}^{1,\frac{N}{q}}(\Omega)$ which satisfies, in the sense of distributions, the inequality  \eqref{eqgen}. Then $u$ is locally H\"older continuous and satisfies for every ball $B\subset\Omega$ the interior bound
\[
|u(x)-u(y)|\leq K|x-y|^\alpha\ ,\forall x,y\in B\ ,
\]
where
\[
\alpha=\min\left\{1-\frac{N}{q\gamma},\frac{\gamma-p}{\gamma-(p-1)}\right\},
\]
and $K$ depends on $p,q,\gamma,N,\nu,\Omega,\|f\|_{\mathcal{L}^{1,\frac{N}{q}}(\Omega)}$ and $\lambda u^+\in \mathcal{L}^{1,\frac{N}{q}}(\Omega)$.
\end{thm}
\begin{proof}
\textit{Step 1: Local H\"older regularity}. Let $x_0\in\Omega$ and $B_r=B_r(x_0)$ be a ball such that the twice bigger ball $B_{2r}(x_0)\subset\Omega$. We apply Lemma \ref{cac} and obtain the estimate
\[
\int_{B_r}|\nabla u|^\gamma\,dx\leq Kr^{N-s},
\]
where $s=\max\left\{\frac{N}{q},\frac{\gamma}{\gamma-(p-1)}\right\}$ and $K$ depends on $\gamma,p,N,\nu,q,c_H$ and $\|f+\lambda u^+\|_{\mathcal{L}^{1,\frac{N}{q}}(\Omega)}$. By the H\"older's inequality we can write
\[
\int_{B_r}|\nabla u|\,dx\leq \left(\int_{B_r}|\nabla u|^\gamma\,dx\right)^\frac1\gamma(\mathrm{Vol}(B_r))^{1-\frac1\gamma},
\]
and hence deduce that for a possibly different constant $\widetilde{K}$ we have
\[
\int_{B_r}|\nabla u|\,dx\leq \widetilde{K}r^{N-\frac{s}{\gamma}}.
\]
If $B_R$ is any ball such that $B_{2R}\subset\Omega$, the same property continues to hold for any smaller ball $B_r\subset B_R$. We apply \cite[Theorem 7.19]{GT} and conclude that $u$ is H\"older continuous in $B_R$ with the following explicit H\"older exponent
\[
\alpha=1-\frac{s}{\gamma}=\min\left\{1-\frac{N}{q\gamma},1-\frac{1}{\gamma-(p-1)}\right\}=\min\left\{1-\frac{N}{q\gamma},\frac{\gamma-p}{\gamma-(p-1)}\right\}
\]
and
\[
|u(x)-u(y)|\leq K|x-y|^{\alpha}\ ,\forall x,y\in B_R\ .
\]
In particular, the last estimate holds for any couple of points $x,y$ belonging to some ball $B_R$ such that $B_{2R}\subset\Omega$.\\

\textit{Step 2: Uniform H\"older estimates}. It is enough to argue through the same path outlined in Step 2 of \cite[Theorem 3.1]{DP}.
\end{proof}
The previous result proves that $u$ is uniformly locally H\"older continuous in a domain $\Omega\subset\R^N$. To prove the global H\"older regularity one needs some extra smoothness assumptions on the boundary of $\Omega$. More precisely, we say that an open bounded subset $\Omega\subset\R^N$ has Lipschitz boundary if in a neighborhood of each $x_0\in\partial\Omega$ the domain $\Omega$ can be represented  as the subgraph of a Lipschitz function of $(N-1)$-variables.\\
Still, we need that $\Omega$ satisfies the uniform interior sphere condition with radius $r>0$, i.e. for every $x_0\in\partial\Omega$ there exists a unit vector $\zeta(x_0)$ such that $B_r(x_0+r\zeta(x_0))\subset\Omega$. Under these extra coupled conditions on the boundary of the domain, one has the following result taken from \cite[Lemma 2.6]{CDLP}, which is recalled here below for reader's convenience.
\begin{lemma}\label{ext}
Let $\Omega\subset\R^N$ be an open bounded domain with Lipschitz boundary and satisfying the uniform interior sphere condition. Let now $u:\Omega\to\R$ be a continuous function such that there exist $K>0$ and $\alpha\in(0,1)$ such that
\[
|u(x)-u(y)|\leq K|x-y|^\alpha\ ,\forall x,y\in B\ ,B\subset\Omega\ .
\]
Then, $u$ extends up to $\partial\Omega$ as a function verifying
\[
|u(x)-u(y)|\leq M|x-y|^\alpha\ ,\forall x,y\in\overline{\Omega}\ ,
\]
where $M\geq K$ depends on $\alpha,K$ and on $\partial\Omega$.
\end{lemma}
We can then conclude the section with the proof of Theorem \ref{main2}.
\begin{proof}[Proof of Theorem \ref{main2}] The result now follows from Theorem \ref{localHolder} via Lemma \ref{ext}.

\end{proof}

\begin{rem}
Theorem \ref{main1} extends a result for semi-solutions to fully nonlinear equations with bounded right-hand side in the viscosity framework obtained in \cite[Theorem 2.11]{CDLP}, together with the one for the case $p=2$ obtained by A. Dall'Aglio and A. Porretta in \cite{DP} for subsolutions to equations with coercive Hamiltonians and $L^q$ source terms. We emphasize that the H\"older regularity we obtain here goes beyond the (continuous) viscosity solutions' framework, and in particular we recover the same (optimal) exponent found in \cite{CDLP} for $f\in L^\infty$. 
\end{rem}
\begin{rem}
The level of H\"older regularity obtained in Theorem \ref{main1} is the starting point to derive more general quantitative (local) Calder\'on-Zygmund estimates for solutions to these classes of PDEs, as recently started in \cite{CV} for $p=2$. This will be the matter of a future research.
\end{rem}

\subsection{Sharpness of the H\"older regularity and further comments}\label{sharp}
We first observe that the exponents found in the H\"older regularity proved in Theorem \ref{main1} are sharp. To do this, we first recall that if $u$ is radial, i.e. $u(x)=V(|x|)\equiv V(r)$ for some smooth function $V$, if $x$ is such that $V'(|x|)\neq 0$ we have
\begin{equation}\label{radial}
-\Delta_pu(x)=-|V'(r)|^{p-2}\left[(p-1)V''(r)+\frac{N-1}{r}V'(r)\right].
\end{equation}
The result in Theorem \ref{main1} shows that when $q>\frac{N(\gamma-(p-1))}{\gamma}$ the supersolutions to the PDI \[-\Delta_p u\geq -\lambda u+|\nabla u|^\gamma-f(x)\] belong to $C^\alpha$ with $\alpha=\frac{\gamma-p}{\gamma-(p-1)}$. In this case, an example of the optimality concerning the H\"older exponent $\alpha=\frac{\gamma-p}{\gamma-(p-1)}$ can be obtained even in the class of (distributional) solutions to the equation
\[
-\Delta_p u=|\nabla u|^\gamma\text{ in }\Omega\ ,u\in W_0^{1,\gamma}(\Omega)\ ,
\]
i.e. for the equation with vanishing source term $f\equiv0$. It is sufficient to take $\Omega=B_1(0)$ and the function $u(x)=c(|x|^{\frac{\gamma-p}{\gamma-(p-1)}}-1)$ for $c=-\frac{\gamma-(p-1)}{\gamma-p}\left(\frac{(N-1)\gamma-N(p-1)}{\gamma-(p-1)}\right)^{\frac{1}{\gamma-(p-1)}}$, $\gamma>p$ and $1<p<N$.\\
We observe that, though the method does not require any restriction on the order $p>1$, the present analysis could have been restricted to the case $1<p< N$, since the weak solutions belonging to $W^{1,\gamma}$ with $\gamma>p$ and $p>N$ would have been automatically (locally) H\"older continuous by the Sobolev embeddings. This is completely in line with the way of proving the local H\"older regularity in the case of $p$-harmonic functions (i.e. for weak solutions to $-\Delta_pu=0$). Indeed, the case $p>N$ is in general simpler than the cases $p=N$ and $p<N$, which require finer arguments based on the Widman filling-the-hole technique and the Moser iteration respectively, cf \cite{BensoussanBook}. We emphasize once more that in the case $1<p\leq N$, our result covers the case of semi-solutions and not only of solutions to such quasi-linear equations.\\
Moreover, even the order of the H\"older class $\alpha=1-\frac{N}{q\gamma}$ cannot be improved when $\frac{N}{\gamma}<q<\frac{N(\gamma-(p-1))}{\gamma}$, as shown in e.g. \cite[Remark 3.2]{DP} for the case $p=2$. Finally, we stress again that when $\gamma>p$, the gradient term dominates the diffusion at small scales, and hence the equation can be regarded as
\[
|\nabla u|^\gamma\leq -\Delta_pu+\lambda u-f
\]
so that the diffusion plays no role in the derivation of the H\"older regularity, as it can be seen by inspection throughout the proof of Lemma \ref{cac}. This result is thus consistent with the Lipschitz regularity for subsolutions of the first-order equation
\[
|\nabla u|^\gamma\leq \lambda u-f,
\]
as widely discussed in \cite{B}. We conclude by recalling that the above example shows also that the uniqueness does not hold for these weak solutions, and hence this latter property really depends on the formulation of the problem, cf \cite[Remark 3.2]{DP}.
\subsection{Liouville theorems for homogeneous problems}\label{sec;lio}
The first byproduct of our Caccioppoli-type bounds obtained in Remark \ref{simp} are the Liouville-type theorems for supersolutions to nonlinear homogeneous PDIs. Although the ideas behind the proofs contained in this section are not completely new, one remarkable observation is that the derivation of the Liouville property for PDIs with gradient terms does not need any a priori one-side bound on the solution, and it holds for merely distributional supersolutions to general PDI with measurable ingredients. This last feature is in line with the corresponding Liouville properties for solutions to the model equation $-\Delta_p u+|Du|^\gamma=0$ obtained via the Bernstein method, cf \cite{PS,Lions85} for $p=2$ and \cite{VeronJFA} for $p>1$.\\
In addition, we wish to emphasize that in the derivation of this kind of Liouville properties, one usually needs to use different methods to handle the subcritical case for $\gamma$ (i.e. $p-1<\gamma<\frac{N(p-1)}{N-1}$) and the critical regime ($\gamma=\frac{N(p-1)}{N-1}$). Here, our method of proof seems shorter compared to the existing ones and unify the treatment of both regimes. To this aim, we consider the (distributional) inequality
\begin{equation}\label{mainlio}
-\mathrm{div}(\mathcal{A}(x,u,\nabla u))\geq \mathcal{B}_2(x,u,\nabla u)\text{ in }\R^N
\end{equation}
with $\mathcal{B}_2$ satisfying \eqref{B2}. For ease of presentation we decided to reverse the order of the proofs of Theorem \ref{main3} and Corollary \ref{main2}. Thus, we first discuss how to derive Corollary \ref{main2} in the Euclidean setting, then state some comments on the result. This will serve as a guideline to treat the more general case of Riemannian manifolds in Section \ref{sec;riem}.
\begin{proof}[Proof of Corollary \ref{main2}]
Suppose by contradiction that $u$ were not constant on the ball $B_{\widetilde{R}}$ for some $\widetilde{R}>0$. The required contradiction can be obtained exploiting \eqref{alpha<1} in Remark \ref{simp} sending $r\to\infty$ or Lemma \ref{cac} applied with $f=\lambda=0$ and $t=0$, after letting $R\to \infty$, using that $\gamma<\frac{N(p-1)}{N-1}$. The Liouville property in the critical case $\gamma=\frac{N(p-1)}{N-1}$ readily follows through the same path using \eqref{alpha1} after sending $r\to\infty$.
\end{proof}
\begin{rem}
Specializing Corollary \ref{main2} to some well-known operators, we conclude that any distributional supersolution to
\[
-\Delta_pu\geq c_H|\nabla u|^\gamma\text{ in }\R^N
\]
or
\[
-\mathrm{div}\left(|\nabla u|^{k-2}\frac{\nabla u}{\sqrt{1+|\nabla u|^k}}\right),\ k\geq2,\ p=\frac{k}{2}
\]
must be constant provided that $p-1<\gamma\leq \frac{N(p-1)}{N-1}$, while any distributional supersolution to
\[
-\Delta u\geq c_H|\nabla u|^\gamma\text{ in }\R^N
\]
or
\[
-\mathrm{div}\left(\frac{\nabla u}{\sqrt{1+|\nabla u|^2}}\right)\geq c_H|\nabla u|^\gamma\text{ in }\R^N
\]
satisfies the Liouville property provided that $1<\gamma\leq \frac{N}{N-1}$. 
\end{rem}

\begin{rem}\label{p>N}
Throughout this section we restricted to consider the case $1<p<N$. Indeed, if one considers the inequality $-\Delta_pu\geq |\nabla u|^\gamma$ in $\R^N$, we have that $u$ solves also $-\Delta_pu\geq 0$ in $\R^N$, and hence the Liouville property would have been trivial when $p\geq N$ through classical results, see e.g. \cite{SZ} and the next Section \ref{alt}.
\end{rem}
\begin{rem}\label{ref}
The results in Corollary \ref{main2} have been deeply studied in the literature. When the PDI is driven by the Laplacian the result has been proved in e.g. \cite{CaristiMitidieri} for classical solutions, i.e. when $u\in C^2(\R^N)$, satisfying the inequality pointwisely, exploiting that the spherical mean preserves the PDI, reducing then the analysis to an ODE. Alternative test functions method have been extensively used in various papers. For instance, R. Filippucci proved in \cite[Corollary 1]{FilippucciJDE2011} the Liouville property for nonnegative distributional solutions to $-\Delta_p u\geq u^m|\nabla u|^\gamma$ in $\R^N$ under the following assumptions
\[
0<m\leq \frac{N(p-1)}{N-p}-\gamma \frac{N-1}{N-p}\ ,1<p<N\ ,m+\gamma>p-1\ ,
\]
and therefore the pure gradient nonlinearity case $m=0$ is excluded from that result. When $p=2$ and $(N-1)\gamma\leq N(p-1)$, $\gamma>0$, the result was proved through a result similar to Lemma \ref{cac} in \cite{DM} for distributional supersolutions. Different (recent) proofs have been obtained in \cite[Corollary 3]{QuaasDCDS} and \cite[Theorem 2.1]{VeronDuke}. We emphasize once more that most of these proofs require to distinguish the treatment of the subcritical case (e.g. when $\gamma<\frac{N(p-1)}{N-1}$) from the critical one (namely $\gamma=\frac{N(p-1)}{N-1}$). More recent analyses for PDIs on noncompact complete manifolds, without curvature conditions and assuming suitable volume growths, have been carried out in \cite{SunCpaa,Sun}. We remark in passing that our method differs from the ones proposed in the literature.
\end{rem}
\subsubsection{Further remarks and alternative approaches}\label{alt}
We conclude the section by observing that a different proof of the Liouville property given in e.g. \cite[Corollary 1]{FilippucciJDE2011}, \cite[Corollary 3]{QuaasDCDS} when a one-side bound on the solution is in force can be derived as follows. This method has been inspired by earlier results on the subject appeared in the context of Riemannian manifolds in \cite{Chen} and \cite{RS,PRSmem} (see also the lecture notes \cite{PSNotes}), although here we propose a variation of those schemes in view of the (superlinear) character of the gradient term. The basic idea relies on multiplying the inequality by $e^{-u}$, use the chain rule and take the integration of the resulting inequality through the divergence theorem. \\
First, assume by contradiction that $u$ is not constant on the ball $B_{R_0}$ for some $R_0>1$. We proceed by considering the vector field $X=e^{-u}(-|\nabla u|^{p-2}\nabla u)$ and integrating its divergence. As before, the argument can be made more rigorous arguing as in Remark \ref{weakdiv} or \cite{PRSjfa} through a test function argument. On one hand, expanding the divergence, after plugging back the equation $-\Delta_p u\geq c_H|\nabla u|^\gamma$ we deduce
\begin{multline}\label{inequpos}
\int_{B_r}\mathrm{div}(-e^{-u}|\nabla u|^{p-2}\nabla u)\,dx=\int_{B_r}e^{-u}|\nabla u|^p\,dx+\int_{B_r}e^{-u}(-\Delta_p u)\,dx\\
\geq \int_{B_r}e^{-u}(-\Delta_p u)\,dx\geq c_H\int_{B_r}e^{-u}|\nabla u|^\gamma\,dx.
\end{multline}
Applying the divergence theorem we get
\[
\int_{B_r}\mathrm{div}(X)\,dx\leq \int_{\partial B_r}e^{-u}|\nabla u|^{p-1}\,d\mathcal{H}^{N-1}
\]
so that we end up with the inequality
\[
\int_{\partial B_r}e^{-u}|\nabla u|^{p-1}\,d\mathcal{H}^{N-1}\geq c_H\int_{B_r}e^{-u}|\nabla u|^\gamma\,dx.
\]
We then use the H\"older's inequality with the conjugate pairs $(\frac{\gamma}{p-1},\frac{\gamma}{\gamma-(p-1)})$, together with $u\geq0$, to conclude
\begin{align*}
c_H\int_{B_r}e^{-u}|\nabla u|^\gamma\,dx&\leq \int_{\partial B_r}e^{-u}|\nabla u|^{p-1}\,d\mathcal{H}^{N-1}\\
&\leq\left(\int_{\partial B_r}e^{-u}|\nabla u|^\gamma\,d\mathcal{H}^{N-1} \right)^\frac{p-1}{\gamma}\left(\int_{\partial B_r}e^{-u}\,d\mathcal{H}^{N-1}\right)^\frac{\gamma-(p-1)}{\gamma}\\
&\leq \left(\int_{\partial B_r}e^{-u}|\nabla u|^\gamma\,d\mathcal{H}^{N-1} \right)^\frac{p-1}{\gamma}\left(\int_{\partial B_r}\,d\mathcal{H}^{N-1}\right)^\frac{\gamma-(p-1)}{\gamma}.
\end{align*}
We set 
\[
\mu(r)=\int_{B_r}e^{-u}|\nabla u|^\gamma\,dx.
\]
Applying the co-area formula \cite{EG} we deduce
\[
\mu'(r)=\int_{\partial B_r}e^{-u}|\nabla u|^\gamma\,d\mathcal{H}^{N-1},
\]
therefore concluding
\[
\mu'(r)[\mu(r)]^{-\frac{\gamma}{p-1}}\geq c_H\left(\int_{\partial B_r}\,d\mathcal{H}^{N-1}\right)^{-\frac{\gamma-(p-1)}{p-1}}.
\]
We then integrate on $[R,r]$ and get 
\begin{multline*}
\frac{p-1}{\gamma-(p-1)}\left[\frac{1}{(\mu(R))^{\frac{\gamma-(p-1)}{p-1}}}-\frac{1}{(\mu(r))^{\frac{\gamma-(p-1)}{p-1}}}\right]\geq \int_R^r\frac{c_H}{(\text{area}(\partial B_t))^{\frac{\gamma-(p-1)}{p-1}}}\,dt\\=c_HN\omega_N\int_R^r\frac{1}{t^{(N-1)\frac{\gamma-(p-1)}{p-1}}}\,dt=\frac{c_HN\omega_N(p-1)}{N(p-1)-(N-1)\gamma}\left[r^{\frac{N(p-1)-(N-1)\gamma}{p-1}}-R^{\frac{N(p-1)-(N-1)\gamma}{p-1}}\right].
\end{multline*}
Therefore
\[
\frac{p-1}{\gamma-(p-1)}\frac{1}{(\mu(R))^{\frac{\gamma-(p-1)}{p-1}}}\geq\frac{c_HN\omega_N(p-1)}{N(p-1)-(N-1)\gamma}\left[r^{\frac{N(p-1)-(N-1)\gamma}{p-1}}-R^{\frac{N(p-1)-(N-1)\gamma}{p-1}}\right].
\]
We then let $r\to\infty$ and conclude $\mu(R)=0$, which then contradicts our initial hypothesis implying that $u$ must be constant a.e. on $\R^N$.\\
When $\gamma=\frac{N(p-1)}{N-1}$ we instead get the inequality
\[
\frac{p-1}{\gamma-(p-1)}\frac{1}{(\mu(R))^{\frac{\gamma-(p-1)}{p-1}}}\geq c_HN\omega_N \log\left(\frac{r}{R}\right),
\]
which leads again to a contradiction. Note that such an approach involving a vector field weighted with an exponential term allows to recover most of the well-known properties for linear and nonlinear problems without perturbative eikonal terms. For instance, one can deduce that every nonnegative superharmonic function in $\R^2$ must be constant, or the Liouville property for nonnegative classical solutions to $-\mathrm{div}(\nabla u/\sqrt{1+|\nabla u|^2})\geq0$ in $\R^2$, or even that any solution to the inequality $-\Delta_p u\geq0$ in $\R^N$ is constant provided $p\geq N$. This can be done using inequality \eqref{inequpos} and keeping the term $\int_{B_r}e^{-u}|\nabla u|^p\,dx$ instead of $\int_{B_r}e^{-u}|\nabla u|^\gamma\,dx$. Actually, this approach can be applied even to deduce the Liouville property for nonnegative $p$-superharmonic functions on general noncompact Riemannian manifolds under the area-growth condition
\[
\int^{+\infty}\frac{1}{(\mathrm{area}(\partial B_\rho))^{\frac{1}{p-1}}}\,d\rho=+\infty,
\]
see \cite{PSNotes} for further details. We recall that such results have been deduced through other different methods: maximum principle methods have been exploited in \cite{PRS,BC,BG}, capacity methods have been used in \cite{HMV99}, and, finally, probabilistic approaches, through the recurrence properties of the Brownian motion, have been thoroughly discussed in \cite{GriAMS}.\\
\begin{rem}
The same integral argument can be repeated to show that any \textit{distributional solution} to
\[
-\Delta_pu+ c_H|\nabla u|^\gamma=0\text{ in }\R^N,\gamma>p-1\ ,
\]
must be constant when $p-1<\gamma\leq \frac{N(p-1)}{N-1}$. On the other hand, if $\gamma>\frac{N(p-1)}{N-1}$ there exist global non-constant solutions belonging to $W^{1,\gamma}_{\text{loc}}(\R^N)$. Indeed, when $\frac{N(p-1)}{N-1}<\gamma<p$ the function $u(x)=-c|x|^{\frac{p-\gamma}{\gamma-(p-1)}}$ for an appropriate $c>0$ belongs to $W^{1,r}_{\mathrm{loc}}(\R^N)$ for $r<N[\gamma-(p-1)]$ and solves the previous equation in distributional sense. When $\gamma>p$ it is sufficient to consider $u(x)=c|x|^{\frac{\gamma-p}{\gamma-(p-1)}}$ for a suitable constant $c>0$. When $\gamma=p$, one can construct a counterexample through logarithmic-type functions. Hence, the results proved in \cite[Theorem A]{VeronJFA} entail that such distributional solutions cannot be more regular (i.e. of class $C^1$). Actually, such a Liouville theorem has been proved for less regular solutions (continuous, in the viscosity sense) in \cite[Theorem 3.1]{BDL} for a wider class of fully nonlinear second order equations avoiding any differentiation of the PDE as a consequence of gradient bounds via the Ishii-Lions method initiated in \cite{CDLP}. Therefore, this integral argument shows once more that such distributional solutions cannot be more regular (neither continuous). Therefore, as remarked in Section \ref{sharp} for the uniqueness of solutions, the formulation of the problem is really important even for the range of the validity of the Liouville theorem.
\end{rem}
\begin{rem}
The Liouville results of the previous sections transfer to some non-divergent and fully nonlinear elliptic equations. For instance, if the leading operator in divergence form in \eqref{eqgen} is replaced with 
\[
\mathcal{P}^-_{\lambda,\Lambda}(M)=\inf\{-\mathrm{Tr}(AM):\lambda I_N\leq A\leq\Lambda I_N\ ,0<\lambda\leq\Lambda\}\ ,
\]
then
\[
\mathcal{P}^-_{\lambda,\Lambda}(D^2u)\geq c_H|\nabla u|^\gamma\text{ in }\R^N\implies -\Lambda \Delta u\geq c_H|\nabla u|^\gamma\text{ in }\R^N\ ,
\]
and hence the Liouville property follows from Theorem \ref{main2} under the same range for $\gamma$. We expect a similar property would hold for supersolutions to equations driven by linear non-divergent operators assuming suitable asymptotic conditions at infinity on the diffusion coefficients. Still, one expects the Liouville property to hold even for the maximal operator $\mathcal{M}^+_{\lambda,\Lambda}$ (which is defined replacing the $\inf$ with the $\sup$ over the same class of matrices) and thus to second order fully nonlinear uniformly elliptic operators, as partly analyzed in \cite{CGsurvey}. These properties in the non-divergence setting will be the matter of a future research, since these integral methods do not apply to such class of PDIs.
\end{rem}
\begin{rem}
The results related to the Liouville property in Theorem \ref{main2} are sharp with respect to the parameter range $1<\gamma\leq\frac{N(p-1)}{N-1}$, as shown in e.g. \cite{FilippucciJDE2011,DM}. Indeed, the function $u(x)=V(|x|)=c(1+|x|^2)^{-\frac{p-\gamma}{2(\gamma-(p-1))}}$ for a suitable $c>0$ is a bounded non-constant solution to the inequality
\[
-\Delta_pu\geq c_H|\nabla u|^\gamma\text{ in }\R^N
\]
when $\gamma>\frac{N(p-1)}{N-1}$. To see this, set $\delta:=\frac{p-\gamma}{\gamma-(p-1)}$. One can use \eqref{radial} to show that
\begin{align*}
-&\Delta_p u=-|c|^{p-2}\left|\frac{\gamma-p}{\gamma-(p-1)}\right|^{p-2}(1+|x|^2)^{-(\frac{\delta}{2}+1)(p-2)}|x|^{p-2}\\
&\cdot c\left[(p-1)\frac{\gamma-p}{\gamma-(p-1)}(1+|x|^2)^{-(\frac{\delta}{2}+2)}\left(\frac{|x|^2}{\gamma-(p-1)}-1\right)-(N-1)\frac{\gamma-p}{\gamma-(p-1)}(1+|x|^2)^{-(\frac{\delta}{2}+1)}\right]\\
&\geq |c|^{p-2}c\frac{p-\gamma}{\gamma-(p-1)}\left|\frac{\gamma-p}{\gamma-(p-1)}\right|^{p-2}\left(N-1-\frac{p-1}{\gamma-(p-1)}\right)(1+|x|^2)^{-(\frac{\delta}{2}+1)(p-1)}|x|^{p-2}\\
&\geq |c|^{p-2}c\frac{p-\gamma}{\gamma-(p-1)}\left|\frac{\gamma-p}{\gamma-(p-1)}\right|^{p-2}\left(N-1-\frac{p-1}{\gamma-(p-1)}\right)(1+|x|^2)^{-(\frac{\delta}{2}+1)\gamma}|x|^{\gamma}.
\end{align*}
The last inequality can be made greater than or equal to 
\[
c_H|\nabla u|^\gamma=c_H|\nabla V(|x|)|^\gamma=c_H|c|^\gamma\left|\frac{\gamma-p}{\gamma-(p-1)}\right|^{\gamma}(1+|x|^2)^{-(\frac{\delta}{2}+1)\gamma}|x|^{\gamma}
\]
 provided that $c>0$ is a suitable small constant and $\gamma>\frac{N(p-1)}{N-1}$.
\end{rem}

\section{Generalizations to problems arising in sub-Riemannian and Riemannian geometry}
\subsection{Problems modeled on H\"ormander's vector fields}\label{sec;sub}
In this section we briefly discuss some possible generalizations to problems modeled on a frame of vector fields of $\R^N$. Consider a family of smooth, say $C^\infty$, vector fields $\X=\{X_1,...,X_m\}$, $m\leq N$, generating a Carnot groups, and thus satisfying the H\"ormander's rank condition. We define the horizontal gradient as $\nabla_\X u=(X_1u,...,X_mu)$ and the symmetrized horizontal Hessian $(D^2_\X u)_{ij}=\frac{X_iX_ju+X_jX_iu}{2}$, while the horizontal divergence is defined as $\mathrm{div}_\X(\Phi(x))=\sum_{i=1}^mX_i^*\Phi_i$, $X_i^*$ being the formal adjoints of the $X_i$'s. We also denote by $B_r^\X(x)=\{y\in\R^N:\rho<R\}$, where $\rho$ stands for a given homogeneous norm. In this case, it holds
\begin{equation}\label{balls}
\mathrm{Vol}(B_r^\X)=c r^\mathcal{Q}\ ,r>0,c>0,
\end{equation}
$\mathcal{Q}$ being the corresponding homogeneous dimension. We denote by $W^{1,r}_{\X,\mathrm{loc}}(\Omega)=\{u\in L^r_{\mathrm{loc}}(\Omega):X_iu\in L^r_{\mathrm{loc}}(\Omega)\ ,i=1,..,m\}$ the standard (local) horizontal Sobolev space.\\
Let $\Omega$ be a bounded open set in $\R^N$, $N\geq1$. We consider the following degenerate PDI
\begin{equation}\label{eqgendeg}
-\mathrm{div}_\X (\mathcal{A}(x,u,\nabla_\X u))\geq \mathcal{B}_3(x,u,\nabla_\X u)\text{ in }\Omega\ ,
\end{equation}
where $\mathcal{A}:\Omega\times\R\times\R^N\to\R^N$ is a Carath\'eodory function such that \eqref{A} holds and
\begin{equation}\label{B1deg}\tag{B3}
\mathcal{B}_3(x,s,\xi)\geq c_H|\xi|^\gamma-f(x).
\end{equation}
We first give the horizontal counterpart of Lemma \ref{cac}.
\begin{lemma}\label{cacHorm}
Let $\gamma>p-1$, $f\in \mathcal{L}^{1,\frac{\mathcal{Q}}{q}}(\Omega)$, $q\geq1$. Assume that \eqref{A} and \eqref{B1deg} hold. Let $u$ be a distributional supersolution of \eqref{eqgendeg}. Then, for every pair of concentric balls $B_t^\X\subset B_R^\X\subset \Omega$ we have
\[
\int_{B_t^\X}|\nabla_\X u|^\gamma\,dx\leq K\frac{R^\mathcal{Q}}{(R-t)^s}
\]
where $s=\max\left\{\frac{\mathcal{Q}}{q},\frac{\gamma}{\gamma-(p-1)}\right\}$ and $K$ is a constant depending on $\nu,\gamma,q,\mathcal{Q},c_H$ and on $\|f\|_{\mathcal{L}^{1,\frac{\mathcal{Q}}{q}}(B_R^\X)}$. In particular, the result holds when $f\in L^q$, $q>\frac{\mathcal{Q}}{\gamma}$.
\end{lemma}
\begin{proof}
The proof is the same as that in Lemma \ref{cac}, the only difference being the choice of a $C^1$ cut-off function $\eta$ such that $0\leq\eta\leq1$, $\eta\equiv1$ on $B_t^\X$, $\eta\equiv0$ outside $B_R^\X$, $|\nabla_\X \eta|\leq \frac{C}{R-t}$.
\end{proof}
As in the Euclidean case, this leads to the corresponding local H\"older continuity for distributional supersolutions combining Lemma \ref{cacHorm} and applying  \cite[Theorem 1.2]{Lug}, as stated in the next. 
\begin{thm}\label{holderHorm}
Assume \eqref{A} and \eqref{B1deg}, $\gamma>p$, and let $f\in \mathcal{L}^{1,\frac{\mathcal{Q}}{q}}(\Omega)$ for some $q>\frac{\mathcal{Q}}{\gamma}$. Let $u\in W^{1,\gamma}_{\mathrm{loc}}(\Omega)$ which satisfies, in the sense of distributions, the inequality
\[
-\mathrm{div}_\X (\mathcal{A}(x,u,\nabla_\X u))\geq |\nabla_\X u|^\gamma+f(x)\text{ in }\Omega\ .
\]
Then $u$ is locally H\"older continuous with exponent 
\[
\alpha=\min\left\{1-\frac{\mathcal{Q}}{q\gamma},\frac{\gamma-p}{\gamma-(p-1)}\right\},
\]
and $K$ depends on $p,q,N,\nu,\Omega,\|f\|_{\mathcal{L}^{1,\frac{\mathcal{Q}}{q}}(\Omega)}$.
\end{thm}
\begin{rem}
Similarly to what obtained in Corollary \ref{main2}, Lemma \ref{cacHorm} leads to a Liouville property for distributional supersolutions (or solutions) when $\mathcal{B}(\nabla_\X u)=|\nabla_\X u|^\gamma$ and
\[
p-1<\gamma\leq \frac{\mathcal{Q}(p-1)}{\mathcal{Q}-1}
\]
These properties have been studied in detail in e.g. \cite{DM}.
\end{rem}
\subsection{Problems posed on Riemannian manifolds}\label{sec;riem}
We end this section with some extensions of the Liouville property for supersolutions in the context of noncompact geodesically complete Riemannian manifolds. In what follows, $M$ will be a smooth connected, noncompact, complete $N$-dimensional Riemannian manifold, while for a fixed origin $o\in M$, we denote by $r(x)$ the distance function from $o$ and, as above, with $B_r$, $\partial B_r$ the geodesic ball and the sphere of radius $r>0$ centered at $o$. We will assume that $\partial B_r$ is smooth for any $r>0$. As discussed in the introduction of \cite{RS} (see also \cite[Theorem 1.1]{PRSjfa}) this is not much restrictive. Moreover, $\mathrm{Vol}(B_r)$ stands for the Riemannian measure of $B_r$, while $\mathrm{area}(\partial B_r)$ the induced measure of $\partial B_r$. The volume growth of the manifold stands for the growth rate of the function $r\longmapsto \mathrm{Vol}(B_r)$, while the area growth of the manifold for the growth rate of $r\longmapsto \mathrm{area}(\partial B_r)$. We will denote by $d\mathrm{vol}_N$ the canonical Riemannian measure, and by $d\mathrm{vol}_{N-1}$ the corresponding $(N-1)$-Hausdorff measure. \\
It is well-known that the Liouville property for subharmonic functions bounded from above, namely the parabolicity of the manifold, is strictly related with its volume growth, and it holds under the condition
\begin{equation}\label{area}
\int^{+\infty}\frac{1}{(\mathrm{area}(\partial B_t(o)))}\,dt=+\infty,
\end{equation}
see \cite[Theorem 7.5]{GriAMS}. Moreover, these properties are tied up with the recurrence properties of the corresponding Brownian motion on $M$ and the existence of Green functions, cf \cite[Theorem 5.1]{GriAMS}.
This in particular implies that $M=\R^2$ is parabolic, i.e. the one-side Liouville property for subharmonic functions bounded from above holds in the plane.\\
Still, it is known, as discussed in \cite{GriAMS}, that the condition
\begin{equation}\label{vol}
\int^{+\infty}\frac{t}{(\mathrm{Vol}(B_t(o)))}\,dt=+\infty
\end{equation}
is sufficient to derive the one-side property, but it is not necessary. Finally, we emphasize that \eqref{area} always follows from \eqref{vol} by \cite[Proposition 1.3]{RS}, but the converse does not hold in general. Therefore, in the study of the next properties we focus on conditions involving $\mathrm{area}(\partial B_r)$ rather than those in terms of $\mathrm{Vol}(B_r)$ or on the optimal parameters for the exponents $\gamma,p$ appearing in the equation (which is the case when $M=\R^N$).\\
The next result provides a nonlinear version of \eqref{area} for quasi-linear equations with power-growth nonlinearities of Hamilton-Jacobi type. To prove Theorem \ref{main3} we state the following useful result, which generalizes Remark \ref{simp}.
\begin{lemma}\label{prelman}
Let $u$ be a nonconstant distributional supersolution to the inequality 
\begin{equation}\label{hjmanifold2}
-\mathrm{div} (\mathcal{A}(x,u,\nabla u))\geq c_H|\nabla u|^\gamma\text{ in }M.
\end{equation}
Then, there exists $\overline{R}>0$ and a constant $C$ depending on $\nu,p,\gamma,N$ such that for every $r>R\geq\overline{R}$ we have
\begin{equation}
\left(\int_{B_R}|\nabla u|^\gamma\,d\mathrm{vol}_N\right)^{-\frac{\gamma-(p-1)}{p-1}}\geq C\int_R^r\left(\frac{1}{\mathrm{area}(\partial B_t)}\right)^{\frac{\gamma-(p-1)}{p-1}}\,dt\ .
\end{equation}
\end{lemma}
\begin{proof}
Let $X$ be the vector field $X:=\mathcal{A}(x,u,\nabla u)$. The proof follows the same steps of Remark \ref{simp} integrating $\mathrm{div}(X)$ and using the (weak) divergence theorem together with the co-area formula, cf Remark \ref{weakdiv}, but we omit this step for brevity. We thus have
\[
-\int_{B_r}\mathrm{div}(\mathcal{A}(x,u,\nabla u))\,d\mathrm{vol}_N\leq \nu\left(\int_{\partial B_r}|\nabla u|^{\gamma}\,\,d\mathrm{vol}_{N-1}\right)^{\frac{p-1}{\gamma}}(\mathrm{area}(\partial B_r))^{\frac{\gamma-(p-1)}{\gamma}}
\]
We set $\sigma(r)=\int_{B_r}|\nabla u|^\gamma\,dx$. Hence, by the co-area formula it follows that \[\sigma'(r)=\int_{\partial B_r}|\nabla u|^\gamma\,d\mathrm{vol}_{N-1}.\]
We then get
\[
\sigma'(r)\geq C(\nu,\gamma,p,c_H)(\sigma(r))^{\frac{\gamma}{p-1}}(\mathrm{area}(\partial B_r))^{-\frac{\gamma-(p-1)}{p-1}},
\]
which reads equivalently as
\[
\sigma'(r)(\sigma(r))^{-\frac{\gamma}{p-1}}\geq C(\nu,\gamma,p,c_H)(\mathrm{area}(\partial B_r))^{-\frac{\gamma-(p-1)}{p-1}}.
\]
We now integrate on $[R,r]$ and obtain that
\begin{align*}
\frac{p-1}{\gamma-(p-1)}\frac{1}{(\sigma(R))^{\frac{\gamma-(p-1)}{p-1}}}&\geq \frac{p-1}{\gamma-(p-1)}\left(\frac{1}{(\sigma(R))^{\frac{\gamma-(p-1)}{p-1}}}-\frac{1}{(\sigma(r))^{\frac{\gamma-(p-1)}{p-1}}}\right)\\
&\geq C(\nu,\gamma,p,c_H)\int_R^r\left(\frac{1}{\mathrm{area}(\partial B_t)}\right)^{\frac{\gamma-(p-1)}{p-1}}\,dt.
\end{align*}

\end{proof}
\begin{proof}[Proof of Theorem \ref{main3}]
If $u$ were not constant on the ball $B_{\widetilde{R}}$, one can apply Lemma \ref{prelman} and find the desired contradiction using \eqref{areanl}.
\end{proof}
\begin{rem}
The previous result is different from the one in Corollary \ref{main2} and must be understood in the following sense: for $\gamma,p$ fixed, \eqref{areanl} highlights how the area of the geodesic sphere can grow to ensure the validity of the Liouville property on a generic manifold without assuming any curvature restriction. This is in line with the recent analyses carried out in \cite{Gricpam,Sun} and the earlier works \cite{RS,PRSjfa}, where the Liouville properties have been studied in terms of the volume (and not the area) growth of the geodesic balls. We emphasize that in view of \cite[Proposition 1.3]{RS} one has for any $\delta>0$
\[
\int^{+\infty}\left(\frac{r}{(\mathrm{Vol}(B_r(o))}\right)^{\frac{1}{\delta}}=+\infty\implies \int^{+\infty}\left(\frac{1}{(\mathrm{area}(\partial B_r(o))}\right)^{\frac{1}{\delta}}=+\infty,
\]
but the converse is not true in general, unless some curvature conditions are imposed on the manifold, see \cite{RS}.
\end{rem}
\begin{rem}
The same area-growth condition could have been obtained through the scheme outlined in Section \ref{alt} for nonnegative solutions.
\end{rem}
\begin{rem}
Although this result has been widely analyzed in the literature in various different settings, we emphasize that the Liouville property under these area-growth conditions seem new. Moreover, though the approach to derive the Liouville property is mainly inspired by \cite{Lions85,RS}, the results of Theorem \ref{main3} cannot be recovered from these works.
\end{rem}
\begin{rem}
The previous result gives a critical condition in terms of the area of the geodesic spheres for the solvability of the equation in $M$, and can be seen as a generalization of the condition found in \cite{HMV99} in the context of Riemannian manfiolds. It is worth noting that when $M=\R^N$ condition \eqref{area} leads to the restriction $p-1<\gamma\leq\frac{N(p-1)}{N-1}$ found in Section \ref{sec;lio}.
\end{rem} 
\begin{rem}
It is worth remarking that Corollary \ref{main2} leads to the Liouville property for nonnegative solutions to an inequality involving powers of the unknown functions and its gradient of the form
\[
-\Delta_p u\geq u^m|\nabla u|^\gamma\text{ in }\R^N
\]
when
\[
(N-p)m+(N-1)\gamma<N(p-1)\ ,m\geq0\ ,\gamma>1\ ,
\]
cf \cite[Remark 1.1]{FPS} for the case $p=2$. This can be done through the transformation $v=u^b$ for a suitable $b>0$, which reduces the previous inequality to an inequality like $-\Delta_pv\geq c_H|\nabla v|^\gamma$ for $\gamma<\frac{N(p-1)}{N-1}$ for some suitable $c_H>0$ depending on $p,\gamma,m$, and thus allows to exploit Corollary \ref{main2}, see e.g. \cite{VeronDuke,CGsurvey}.
\end{rem}
\begin{rem}
As remarked in \cite[p.489]{RS} or in \cite[Theorem 4.6 and Remark 4.8]{IPS} using the same techniques of the present paper one can obtain similar results for supersolutions to elliptic problems equipped with Neumann boundary conditions posed on manifolds with boundary.
\end{rem}

\end{document}